\documentclass[12pt]{article}
\usepackage[cp1251]{inputenc}

\usepackage{amsfonts}
\usepackage{amsmath,amsthm}
\usepackage{amssymb}
\usepackage{indentfirst}

\usepackage[russian,english]{babel}

\renewcommand{\le}{\leqslant}
\renewcommand{\ge}{\geqslant}

\newcommand{\R}{\mathbb{R}}
\newtheorem{theorem}{Theorem}

\begin{document}
\title{Acute sets}
\author{D. Zakharov}
\date{}
\maketitle

\begin{abstract} A set of points in $\R^d$ is {\it acute}, if any three points from this set form an acute angle. In this note we construct an acute set in $\R^d$ of size at least $2^{d/2}$.  \end{abstract}
A set of points in $\R^d$ is {\it acute}, if any three points from this set form an acute angle.
In 1972 Danzer and Gr\"{u}nbaum \cite{DG} posed the following question: what is the maximum size $f(d)$ of an acute set in $\R^d$? They proved a linear lower bound $f(d) \ge 2d-1$ and conjectured that this bound is tight. However, in 1983 Erd\H os and F\"uredi \cite{EF} disproved this conjecture in large dimensions. They gave an exponential lower bound
$$
f(d) \ge \frac{1}{2} \left ( \frac{2}{\sqrt{3}} \right )^d > 0,5 \cdot 1,154^d.
$$
Their proof is a very elegant application of the probabilistic method. One drawback of their method is that only the existence of an acute set of such size is proven, with no possibility to turn it into an explicit construction.

In 2011 Harangi \cite{H} refined the approach of Erd\H os and F\"uredi and improved their bound to
$$
f(d) \ge c \left ( \sqrt[10]{\frac{144}{23}} \right )^d>c \cdot 1,2^d.
$$

In this note we prove the following recurrent inequality:
\begin{theorem}\label{thm1}
$f(d+2)\ge 2f(d)$.
\end{theorem}
Theorem~\ref{thm1} easily implies the bound
$$
f(d) \ge 2^{\frac d2+1} \ \  \ \ \ \text{for }\ \ \ n\ge 4,
$$
since it is known \cite{H} that $f(4) \ge 8$ and $f(5) \ge 12$.

The proof of Theorem~\ref{thm1} is explicit and allows to construct acute sets effectively.\\

The best known upper bound on $f(d)$ is $f(d)\le 2^d-1$, and follows from the main result of \cite{DG}. Danzer and Gr\"unbaum proved that if a set $S$ of points in $\R^d$ determines only acute and right angles, then $|S|\le 2^d$. Moreover, if $|S|=2^d$, then $S$ must be an affine image of a $d$-dimensional cube.
\\

\begin{proof}[Proof of Theorem~\ref{thm1}]

 The scalar product of two vectors $u,v$ is denoted by $\langle u,v\rangle$. Take the largest acute set $X \subset \R^d$, $|X| = f(d)$. Put $$s:=\min\{\langle y-x,z-x\rangle: x, y, z \in X, x \neq y, x \neq z\}.$$ Since the set $X$ is acute, we have $s>0$, and we can take a positive number $r$ such that $4r^2 < s$.

 For each $x \in X$ choose a point $\phi(x)$ on a circle of radius $r$ with center in the origin so that all points $\pm \phi(x)$ are distinct.
Construct a set $Y \subset \R^{d+2}$:
$$
Y := \{ (x, \pm \phi(x)), x \in X\}
$$
Clearly, $|Y|=2|X|=2f(d)$. Our aim is to prove that $Y$ is acute. Take three distinct points $\tilde x,\tilde y, \tilde z \in Y$, where
$$
\tilde x = (x, a\phi(x)),\ \ \ \tilde y = (y, b\phi(y)),\ \ \ \tilde z = (z, c\phi(z)),\ \ \ a,b,c\in\{\pm 1\}.
$$

Suppose that $x \neq y$ and $x \neq z$. Then
\begin{equation} \label{1}
\langle\tilde y - \tilde x,\tilde z - \tilde x\rangle=\langle y-x,z-x\rangle+\langle b\phi(y)-a\phi(x),c\phi(z)-a\phi(x)\rangle.
\end{equation}
The first scalar product on the right hand side is at least $s$ by the definition of $s$, while the second scalar product is at most $4r^2$. By the choice of $r$, the sum of these two scalar products is positive, which means that the angle $(\tilde y,\tilde x, \tilde z)$ is acute.

Suppose that $x=y$ (the case $x=z$ is treated in the same way). We have $a+b=0$, so
$$
\langle\tilde y - \tilde x,\tilde z - \tilde x\rangle=\langle b\phi(y)-a\phi(x),c\phi(z)-a\phi(x)\rangle=\langle 2a\phi(x), a\phi(x)-c\phi(z)\rangle=2\big(\|\phi(x)\|^2\pm \langle\phi(x),\phi(z)\rangle\big) >0,
$$
because $\phi(x) \neq \pm\phi(z)$. Thus, the angle $(\tilde y,\tilde x, \tilde z)$ is acute in this case as well.

We conclude that each angle in $Y$ is acute. The theorem is proved.
\end{proof}

 \textsc{Acknowledgements}: We would like to thank Andrey Kupavskii for discussions that helped to improve the main result, as well as for his help in preparing this note. We would also like to thank Prof. Raigorodskii for introducing us to this problem and for his constant encouragement.

\end{document}